\documentclass[12pt,english]{article}
\usepackage[T1]{fontenc}
\usepackage[latin9]{inputenc}
\usepackage[letterpaper]{geometry}
\usepackage{array}
\usepackage{float}
\usepackage{textcomp}
\usepackage{mathrsfs}
\usepackage{multirow}
\usepackage{amsthm}
\usepackage{amsmath}
\usepackage{amssymb}
\usepackage{graphicx}
\usepackage{setspace}
\usepackage[numbers]{natbib}
\makeatletter

\newcommand{\lyxmathsym}[1]{\ifmmode\begingroup\def\b@ld{bold}
  \text{\ifx\math@version\b@ld\bfseries\fi#1}\endgroup\else#1\fi}

\providecommand{\tabularnewline}{\\}

\numberwithin{equation}{section}
\numberwithin{figure}{section}
\newcommand{\lyxaddress}[1]{
\par {\raggedright #1
\vspace{1.4em}
\noindent\par}
}
  \theoremstyle{definition}
  \newtheorem{defn}{\protect\definitionname}[section]
  \theoremstyle{plain}
  \newtheorem{lem}{\protect\lemmaname}[section]
  \theoremstyle{plain}
  \newtheorem{prop}{\protect\propositionname}[section]
  \theoremstyle{remark}
  \newtheorem{rem}{\protect\remarkname}[section]
  \theoremstyle{plain}
  \newtheorem{thm}{\protect\theoremname}[section]
  \theoremstyle{definition}
  \newtheorem{example}{\protect\examplename}[section]

\makeatother

\usepackage{babel}
  \providecommand{\definitionname}{Definition}
  \providecommand{\examplename}{Example}
  \providecommand{\lemmaname}{Lemma}
  \providecommand{\propositionname}{Proposition}
  \providecommand{\remarkname}{Remark}
\providecommand{\theoremname}{Theorem}

\begin{document}

\title{A Bayesian nonparametric chi-squared goodness-of-fit test}

\maketitle
\begin{center}
~ \\
Reyhaneh Hosseini and Mahmoud Zarepour
\par\end{center}

\lyxaddress{\begin{center}
Department of Mathematics and Statistics\\
University of Ottawa
\par\end{center}}
\begin{abstract}
~The Bayesian nonparametric inference and Dirichlet process are popular
tools in Bayesian statistical methodologies. In this paper, we employ
the Dirichlet process in a hypothesis testing to propose a Bayesian
nonparametric chi-squared goodness-of-fit test. In our new Bayesian
nonparametric approach, we consider the Dirichlet process as the prior
for the distribution of the data and carry out the test based on the
Kullback-Leibler distance between the updated Dirichlet process and
the hypothesized distribution. We prove that this distance asymptotically
converges to the same chi-squared distribution as the classical frequentist's
chi-squared test does. Moreover, a Bayesian nonparametric chi-squared
test of independence for a contingency table is described. In addition,
by computing the Kullback-Leibler distance between the Dirichlet process
and the hypothesized distribution, a method to obtain an appropriate
concentration parameter for the Dirichlet process is presented.
\end{abstract}

\subsubsection*{MSC 2010: \textmd{Primary 62G20; secondary 62G10.}}

\section{Introduction}

~The Bayesian nonparametric plays a crucial role in statistical inference.
The Dirichlet process perhaps is the most popular prior in Bayesian
nonparametric statistics and it has been applied in many different
areas of statistical inference. The most common applications of Dirichlet
process are in density estimation and clustering via mixture models.
See for instance, \citet{neal1992bayesian}, \citet{lo1984class}
and \citet{escobar1995bayesian}. In this paper, we suggest a Bayesian
nonparametric chi-squared goodness-of-fit test based on the Kullback-Leibler
distance between the posterior Dirichlet process and the hypothesized
distribution.

There are many one-sample and two-sample parametric goodness-of-fit
tests in the literature. See for example, \citet{d1986goodness} for
a review. The chi-squared test examines whether the data has a specified
distribution $F_{0}$, i.e., the null hypothesis is given as $H_{0}:F=F_{0}$
where $F_{0}$ is the true distribution for the observed data. Some
extensions of chi-squared goodness-of-fit test to Bayesian model assessment
where the test statistic is based on the posterior distribution, are
described by \citet{johnson2004bayesian} and \citet{johnson2007bayesian}.

In Bayesian nonparametric inference, there are two strategies of goodness-of-fit
test. The first strategy considers a prior for the true distribution
of data and constructs the test based on the distance between the
posterior distribution and the proposed one. For example, \citet{muliere1998approximating},
\citet{swartz1999nonparametric}, \citet{al2013bayesian,al2014goodness}
considered the Dirichlet process prior and the Kolmogorov distance.
\citet{al2014goodness} and \citet{labadi2014two} carried out a goodness-of-fit
test and a two-sample goodness-of-fit test, respectively by considering
the Dirichlet process as a prior and the test statistic based on the
Kolmogorov distance. \citet{viele2000evaluating} used the Dirichlet
process and the Kullback-Leibler distance for testing the discrete
distributions. \citet{hsieh2013nonparametric} considered the Polya
tree model as the prior and measured the Kullback-Leibler distance
for testing the continuous distributions.

The second strategy is conducted by embedding the hypothesized model
$H_{0}$ in an alternative model $H_{1}$ and placing a prior on that.
To examine the hypothesized model, the Bayes factor is used as a measure
of evidence against the hypothesized model. For example, \citet{carota1994bayes}
and \citet{florens1996bayesian} used a Dirichlet process prior for
the alternative model. \citet{tokdar2011bayesian} carried out a Bayesian
test for normality by considering a Dirichlet process mixture for
the alternative model. Some authors used other Bayesian nonparametric
priors. For instance, \citet{holmes2015two} described a Bayesian
nonparametric two sample hypothesis testing based on a Polya tree
prior. In order to test for the normal distribution, \citet{berger2001bayesian}
considered a mixture of Polya trees for the alternative model distribution,
while \citet{verdinelli1998bayesian} suggested a mixture of Gaussian
processes.

Our new proposed chi-squared goodness of fit test is based on the
first approach discussed above. We consider a Dirichlet process prior
for the distribution of the observed data and define the chi-squared
test statistic based on the Kullback-Leibler distance between the
Dirichlet process posterior and the hypothesized distribution. In
fact, in our Bayesian nonparametric approach, the test proceeds by
constructing the chi-squared test statistic based on the distance
between the observed probabilities obtained by the Dirichlet process
posterior and the expected probabilities. Indeed, instead of counting
the observed frequencies in each bin, we place a prior on the distribution
of the data. The probability of each bin is obtained by the exact
posterior probability of that bin. Then, our new test statistic compares
the posterior probabilities with the probabilities under the null
hypothesis. In this procedure, based on the suggested Dirichlet prior,
we know the exact distribution of the test statistic. Using a similar
approach, we also determine an appropriate concentration parameter
for the Dirichlet process which is required to decide on an appropriate
prior.

The outline of the paper is organized as follows. In Section 2, we
give an essential background on Dirichlet process and its properties.
In Section 3, we briefly review the definition of the Kullback-Leibler
divergence. Following this, we obtain the Kullback-Leibler distance
between the Dirichlet process and a continuous distribution and compute
its mean and variance. Section 4 discusses a Bayesian nonparametric
chi-squared goodness-of-fit test based on the Kullback-Leibler distance
between the Dirichlet process posterior and the hypothesized distribution.
In Section 5, we extend our suggested chi-squared test to present
a Bayesian nonparametric chi-squared test of independence of two random
variables. We also describe a method to obtain an appropriate concentration
parameter based on the Kullback-Leibler distance between the Dirichlet
process and the proposed distribution. Simulation studies of the tests
with a data illustration appear in Section 6. In the final section,
we conclude with a brief discussion and the Appendix contains the
theoretical results.

\section{Dirichlet Process }

In this section, we review the construction, various properties and
some series representations of the Dirichlet process. The Dirichlet
process was initially formalized by \citet{fb1973} for general Bayesian
statistical modeling as a distribution over probability distributions.
\begin{defn}
(\citet{fb1973}) Let $\mathscr{X}$ be a set, $\mathscr{A}$ be a
$\sigma-$field of subsets of $\mathscr{X}$, $H$ be a probability
measure on $(\mathscr{X,A})$ and $\alpha>0$. A random probability
measure $P$ with parameters $\alpha$ and $H$ is called a Dirichlet
process (denoted by $P\sim DP(\alpha H)$ ) on $(\mathscr{X,A})$
if for any finite measurable partition $\{A_{1},\ldots,A_{k}\}$ of
$\mathscr{X}$, the joint distribution of the random variables $P(A_{1}),\ldots,P(A_{k})$
is a k-dimensional Dirichlet distribution with parameters $\alpha H(A_{1}),\ldots,\alpha H(A_{k})$,
where $k\geq2$.
\end{defn}
We assume that if $H(A_{k})=0$, then $P(A_{k})=0$ with probability
one. Then, a Dirichlet process is parameterized by $\alpha$ and $H$
which are called the concentration parameter and the base distribution,
respectively. The base distribution is also the mean of the Dirichlet
process, i.e., for any measurable set $A\subset\mathscr{X}$, $E\left(P(A)\right)=H(A).$
One of the most remarkable properties of the Dirichlet process is
that it satisfies the conjugacy property. Let $X_{1},\ldots,X_{m}$
be an i.i.d. sample from $P\thicksim\textrm{\textrm{\ensuremath{DP}}}(\alpha H)$.
The posterior distribution of $P$ given $X_{1},\ldots,X_{m}$ is
a Dirichlet process with parameters
\begin{equation}
\alpha_{m}^{*}=\alpha+m\,\textrm{ and }\, H_{m}^{*}=\frac{\alpha}{\alpha+m}H+\frac{m}{\alpha+m}\frac{\sum_{i=1}^{m}\delta_{X_{i}}}{m}\label{eq:h&a}
\end{equation}
 and denoted by $P_{m}^{*}=\left(P\mid X_{1},\ldots,X_{m}\right)\thicksim DP(\alpha_{m}^{*}H_{m}^{*})$,
where $\delta_{X}(\cdot)$ is the Dirac measure, i.e., $\delta_{X}(A)=1$
if $X\in A$ and $0$ otherwise.

As it is seen in \eqref{eq:h&a}, the posterior base distribution
$H_{m}^{*}$ is a weighted average of $H$ and the empirical distribution
$F_{m}=\frac{\sum_{i=1}^{m}\delta_{X_{i}}}{m}$. Thus, for large values
of $\alpha$, $H_{m}^{*}\overset{a.s.}{\rightarrow}H$. On the other
hand, as $\alpha\rightarrow0$ or as the number of observations $m$
grows large, $H_{m}^{*}$ becomes non-informative in the sense that
$H_{m}^{*}$ is just given by the empirical distribution and is a
close approximation of the true underlying distribution of $X_{i}$,
$i=1,\ldots,m$. This confirms the consistency property of the Dirichlet
process, i.e., the posterior Dirichlet process approaches the true
underlying distribution. For a discussion about the consistency property
of Dirichlet process, see \citet{ghosal2010dirichlet} and \citet{james2008large}.

A sum representation of Dirichlet process is presented by \citet{fb1973}
based on the work of \citet{ferguson1972representation}. Specifically,
let $(\theta_{i})_{i\geq1}$ be a sequence of i.i.d. random variables
with common distribution $H$ and $(E_{k})_{k\geq1}$ be a sequence
of i.i.d. random variables from the exponential distribution with
mean 1. If $\Gamma_{i}=E_{1}+\cdots+E_{i}$ and $(\Gamma_{i})_{i\geq1}$
are independent from $(\theta_{i})_{i\geq1}$, then,

\begin{equation}
P=\underset{i=1}{\overset{\infty}{\sum}}\frac{L^{-1}(\Gamma_{i})}{\overset{\infty}{\underset{i=1}{\sum}}L^{-1}(\Gamma_{i})}\delta_{\theta_{i}}=\underset{i=1}{\overset{\infty}{\sum}}p_{i}\delta_{\theta_{i}}\label{eq:pFer}
\end{equation}
is a Dirichlet process with parameters $\alpha$ and $H$ where $L(x)=\alpha\int_{x}^{\infty}t^{-1}e^{-t}dt,\: x>0$
and $L^{-1}(y)=\inf\{x>0:L(x)\geq y\}.$ \citet{ishwaran2002exact}
introduced a finite sum approximation for the Dirichlet process which
is easier to work with. Let $\mathbf{p}=(p_{1,n},\ldots,p_{n,n})$
has a Dirichlet distribution with parameters $(\alpha/n,\ldots,\alpha/n)$
denoted by $\textrm{Dir}(\alpha/n,\ldots,\alpha/n)$ and $(\theta_{i})_{1\leq i\leq n}$
be a sequence of i.i.d. random variables with distribution $H$ and
independent of $(p_{i,n})_{1\leq i\leq n}$. Also, let $(\mathcal{G}_{i,n})_{1\leq i\leq n}$
be i.i.d. random variables from $\textrm{Gamma}(\alpha/n,1)$ distribution
and $p_{i,n}=\mathcal{G}_{i,n}/\mathcal{G}_{n}$, where $\mathcal{G}_{n}=\mathcal{G}_{1,n}+\cdots+\mathcal{G}_{n,n}.$
Then,

\begin{equation}
P_{n}=\underset{i=1}{\overset{n}{\sum}}p_{i,n}\delta_{\theta_{i}}=\underset{i=1}{\overset{n}{\sum}}\frac{\mathcal{G}_{i,n}}{\mathcal{G}_{n}}\delta_{\theta_{i}}\label{eq:finite}
\end{equation}
is called a finite-dimensional Dirichlet process and approximates
the Ferguson's Dirichlet process weakly. Another finite sum representation
of the Dirichlet process with monotonically decreasing weights is
presented in \citet{zarepour2012rapid}. Specifically, let $(\theta_{i})_{1\leq i\leq n}$
be a sequence of i.i.d. random variables with values in $\mathscr{X}$
and common distribution $H$ and independent of $(\Gamma_{i})_{1\leq i\leq n+1}$.
Let $X_{n}\sim\textrm{Gamma}(\alpha/n,1)$ and define

\[
G_{n}(x)=\Pr(X_{n}>x)=\frac{1}{\Gamma(\alpha/n)}\int_{x}^{\infty}t^{(\alpha/n)-1}e^{-t}dt
\]
and

\[
G_{n}^{-1}(y)=\inf\{x:G_{n}(x)\geq y\}.
\]
Then, as $n\rightarrow\infty,$
\begin{equation}
P_{n}=\underset{i=1}{\overset{n}{\sum}}\frac{G_{n}^{-1}(\frac{\Gamma_{i}}{\Gamma_{n+1}})}{\overset{n}{\underset{i=1}{\sum}}G_{n}^{-1}(\frac{\Gamma_{i}}{\Gamma_{n+1}})}\delta_{\theta_{i}}\overset{a.s.}{\rightarrow}P=\underset{i=1}{\overset{\infty}{\sum}}\frac{L^{-1}(\Gamma_{i})}{\overset{\infty}{\underset{i=1}{\sum}}L^{-1}(\Gamma_{i})}\delta_{\theta_{i}}.
\end{equation}
If we define
\begin{equation}
p_{i,n}=\frac{G_{n}^{-1}(\frac{\Gamma_{i}}{\Gamma_{n+1}})}{\overset{n}{\underset{i=1}{\sum}}G_{n}^{-1}(\frac{\Gamma_{i}}{\Gamma_{n+1}})},\label{eq:pi}
\end{equation}
then, $P_{n}$ can be written as
\begin{equation}
P_{n}=\underset{i=1}{\overset{n}{\sum}}p_{i,n}\delta_{\theta_{i}}.\label{eq:pd}
\end{equation}
This finite sum representation converges almost surely to Ferguson's
representation and empirically converges faster than the other representations.
For other sum representations of Dirichlet process, see for example,
\citet{sethuraman1991constructive} and \citet{bondesson1982simulation}.
In the next section, we will discuss computing the Kullback-Leibler
distance between the Dirichlet process and a continuous distribution
and its mean and variance.

\section{Kullback-Leibler distance between the Dirichlet process and a continuous
distribution}

The Kullback-Leibler distance that measures the distance between two
distributions introduced by \citet{kullback1951information}. Suppose
$\mathcal{P}$ and $\mathcal{Q}$ are two probability measures for
discrete random variables on a measurable space $(\Omega,\mathrm{\mathcal{F})}$.
The Kullback-Leibler divergence between $\mathcal{P}$ and $\mathcal{Q}$
is defined as

\begin{equation}
D_{KL}(\mathcal{P}\parallel\mathcal{Q})=\underset{i}{\sum}\mathcal{P}(i)\log\left(\frac{\mathcal{P}(i)}{\mathcal{Q}(i)}\right).
\end{equation}
For continuous probability measures $\mathcal{P}$ and $\mathcal{Q}$
with $\mathcal{P}$ absolutely continuous with respect to $\mathcal{Q}$,
the Kullback-Leibler distance is written as
\[
D_{KL}(\mathcal{P}\parallel\mathcal{Q})=\int\log\left(\frac{d\mathcal{P}}{d\mathcal{Q}}\right)d\mathcal{P}
\]
where $\frac{d\mathcal{P}}{d\mathcal{Q}}$ is the Radon-Nikodym derivative
of $\mathcal{P}$ with respect to $\mathcal{Q}$. Let $\mathcal{P}\ll\lambda$
and $\mathcal{Q}\ll\lambda$ where $\lambda$ is the Lebesgue measure.
If the densities of $\mathcal{P}$ and $\mathcal{Q}$ with respect
to Lebesgue measure are denoted by $p(x)$ and $q(x),$ respectively,
then the Kullback-Leibler distance is written as
\begin{equation}
D_{KL}(\mathcal{P}\parallel\mathcal{Q})=\underset{\mathbb{R}}{\int}p(x)\log\left(\frac{p(x)}{q(x)}\right)dx.
\end{equation}
We compute the distance between the random distribution $P$ from
a Dirichlet process $DP(\alpha H)$ and a continuous distribution
$F$ with density $f(x)$. Since $P$ is a discrete measure and $F$
is continuous, we estimate the density $f(x)$ by its histogram estimator
on a partitioned space. Also, since the Kullback-Leibler distance
is not symmetric, we compute both distances $\textmd{D}_{KL}(P\parallel F)$
and $\textmd{D}_{KL}(F\parallel P)$.
\begin{lem}
\label{Lemma}Let $H$ and $F$ be two distributions defined on the
same space $\mathcal{X}$ and $P_{n}=\sum_{i=1}^{n}p_{i,n}\delta_{\theta_{i}}$
be a random distribution as defined in \eqref{eq:finite}, i.e., $\theta_{1},\ldots,\theta_{n}$
are i.i.d. generated from $H$ with corresponding order statistics
$\theta_{(1)},\ldots,\theta_{(n)}$. We have

\begin{equation}
D_{KL}(P_{n}\parallel F)=-\mathcal{H}(\mathbf{p})-\overset{n}{\underset{i=1}{\sum}}p_{i,n}\log(q_{i})\label{eq:d1}
\end{equation}
and

\begin{equation}
D_{KL}(F\parallel P_{n})=-\mathcal{H}(\mathbf{q})-\overset{n}{\underset{i=1}{\sum}}q_{i}\log(p_{i,n})\label{eq:d2}
\end{equation}
where $\mathcal{H}(\mathbf{p})=-\overset{n}{\underset{i=1}{\sum}}p_{i,n}\log(p_{i,n})$
is the entropy of $P_{n}$ and $\mathcal{H}(\mathbf{q})=-\overset{n}{\underset{i=1}{\sum}}q_{i}\log(q_{i})$
with $q_{i}=\frac{\triangle F(x_{i})}{\triangle x_{i}}$. \end{lem}
\begin{proof}
See the Appendix.

The mean and the variance of the Kullback-Leibler divergences \eqref{eq:d1}
and \eqref{eq:d2} are given in the following Proposition and Remark.\end{proof}
\begin{prop}
\label{prob d} Let $H$ and $F$ be distributions defined on the
same space $\mathcal{X}$ and $P_{n}=\sum_{i=1}^{n}p_{i,n}\delta_{\theta_{i}}$
be a random distribution as defined in \eqref{eq:finite}, i.e., $\theta_{1},\ldots,\theta_{n}$
are i.i.d. generated from $H$ with corresponding order statistics
$\theta_{(1)},\ldots,\theta_{(n)}$. Then, the mean and the variance
of the Kullback-Leibler divergence \eqref{eq:d1} are given as

\begin{equation}
E(D_{KL}(P_{n}\parallel F))=n\left(\psi\left(\frac{\alpha}{n}+1\right)-\psi(\alpha+1)\right)-\frac{1}{n}\overset{n}{\underset{i=1}{\sum}}\log(q_{i})\label{eq:ed1}
\end{equation}
and
\begin{eqnarray}
Var(D_{KL}(P_{n} & \parallel & F))=\overset{n}{\underset{i=1}{\sum}}\left\{ Var\left(p_{i,n}\log(p_{i,n})\right)+\left(\log(q_{i})\right)^{2}Var(p_{i,n})\right\} \nonumber \\
 &  & -2\overset{n}{\underset{i=1}{\sum}}\left\{ \log(q_{i})Cov\left(p_{i,n}\log(p_{i,n}),p_{i,n}\right)\right\} \nonumber \\
 &  & +2\overset{}{\underset{i<j}{\sum}}\left\{ Cov\left(p_{i,n}\log(p_{i,n}),p_{j,n}\log(p_{j,n})\right)+\log(q_{i})\log(q_{j})Cov\left(p_{i,n},p_{j,n}\right)\right\} \nonumber \\
 &  & -4\underset{i<j}{\sum}\left\{ \log(q_{i})Cov\left(p_{i,n}\log(p_{i,n}),p_{j,n}\right)\right\} ,\label{eq:vd1}
\end{eqnarray}
respectively, where {\footnotesize{
\begin{eqnarray*}
Var(p_{i,n}) & = & \frac{n-1}{n^{2}(\alpha+1)},\\
Cov\left(p_{i,n},p_{j,n}\right) & = & \frac{-1}{n^{2}(\alpha+1)},\\
Var\left(p_{i,n}\log(p_{i,n})\right) & = & \frac{(\alpha/n)+1}{n(\alpha+1)}\left(\psi_{1}\left(\frac{\alpha}{n}+2\right)-\psi_{1}(\alpha+2)+\left[\psi\left(\frac{\alpha}{n}+2\right)-\psi(\alpha+2)\right]^{2}\right)\\
 &  & -\left(\psi(\frac{\alpha}{n}+1)-\psi(\alpha+1)\right)^{2},\\
Cov\left(p_{i,n}\log(p_{i,n}),p_{i,n}\right) & = & \frac{(\alpha/n)+1}{n(\alpha+1)}\left(\psi\left(\frac{\alpha}{n}+2\right)-\psi(\alpha+2)\right)-\frac{1}{n}\left(\psi\left(\frac{\alpha}{n}+1\right)-\psi(\alpha+1)\right),\\
Cov\left(p_{i,n}\log(p_{i,n}),p_{j,n}\right) & = & \frac{\alpha}{n^{2}(\alpha+1)}\left(\psi\left(\frac{\alpha}{n}+1\right)-\psi(\alpha+2)\right)-\frac{1}{n}\left(\psi\left(\frac{\alpha}{n}+1\right)-\psi(\alpha+1)\right),\\
Cov\left(p_{i,n}\log(p_{i,n}),p_{j,n}\log(p_{j,n})\right) & = & \frac{\alpha}{n^{2}(\alpha+1)}\left\{ \left(\psi\left(\frac{\alpha}{n}+1\right)-\psi(\alpha+2)\right)^{2}-\frac{\alpha\psi_{1}(\alpha+2)}{n^{2}(\alpha+1)}\right\} \\
 &  & -\left(\psi\left(\frac{\alpha}{n}+1\right)-\psi(\alpha+1)\right)^{2}
\end{eqnarray*}
}}and $\psi(\alpha)=\frac{d\ln\left(\Gamma(\alpha)\right)}{d\alpha}$
and $\psi_{1}(\alpha)=\frac{d^{2}\ln\left(\Gamma(\alpha)\right)}{d\alpha^{2}}=\frac{d\psi(\alpha)}{d\alpha}$
are called digamma and trigamma functions, respectively.\end{prop}
\begin{proof}
The proof is given in Appendix.\end{proof}
\begin{rem}
\label{remark d} Let $H$ and $F$ be two distributions defined on
the same space $\mathcal{X}$ and $P_{n}=\sum_{i=1}^{n}p_{i,n}\delta_{\theta_{i}}$
be the finite dimensional distribution as defined in \eqref{eq:finite},
in which $\theta_{1},\ldots,\theta_{n}$ are i.i.d. generated from
$H$ with corresponding order statistics $\theta_{(1)},\ldots,\theta_{(n)}$.
The mean and the variance of the Kullback-Leibler divergence \eqref{eq:d2}
can be obtained as
\begin{equation}
E(D_{KL}(F\parallel P_{n}))=-\mathcal{H}(\mathbf{q})-\left(\psi\left(\frac{\alpha}{n}\right)-\psi(\alpha)\right)\label{eq:ed2}
\end{equation}
and

\begin{equation}
Var(D_{KL}(F\parallel P_{n}))=\overset{n}{\underset{i=1}{\sum}}q_{i}^{2}\psi_{1}\left(\frac{\alpha}{n}\right)-\psi_{1}(\alpha),\label{eq:vd2}
\end{equation}
respectively.\end{rem}
\begin{proof}
The proof is given in Appendix.
\end{proof}

\section{Bayesian nonparametric chi-squared goodness-of-fit test}

The null hypothesis of the goodness-of-fit test is given as $H_{0}:F=F_{0}$
where $F$ is the true underlying distribution of the observed data
and $F_{0}$ is some specified distribution. Pearson\textquoteright{}s
chi-squared goodness of fit test proceeds by partitioning the sample
space into $k$ non-overlapping bins and comparing the observed counts
with the expected counts under the null hypothesis for each bin. Suppose
$X_{1},\ldots,X_{m}$ is a sample of size $m$ from the distribution
$F$. Let $O_{i}$ and $E_{i},\: i=1,\ldots,k$ denote the observed
counts and the expected counts under the hypothesized distribution
$F_{0}$ for bin $k$, respectively. The Pearson's goodness-of-fit
test statistic is defined as

\begin{equation}
X^{2}=\underset{i=1}{\overset{k}{\sum}}\frac{(O_{i}-E_{i})^{2}}{E_{i}}\label{eq:chi-s}
\end{equation}
and $X^{2}$ asymptotically converges to a chi-squared distribution
with $k-1$ degrees of freedom. To derive a counter part Bayesian
nonparametric test statistic similar to $X^{2}$, we consider a Dirichlet
process with parameters $\alpha$ and $H=F_{0}$ as a prior for the
true distribution of data, i.e., $X_{1},\ldots,X_{m}\sim P$ where
$P\thicksim DP(\alpha H)$. Then, given $X_{1},\ldots,X_{m}$, the
posterior distribution of $P$ is a Dirichlet process $P_{m}^{*}=\left(P\mid X_{1},\ldots,X_{m}\right)\thicksim DP(\alpha_{m}^{*}H_{m}^{*})$
where $\alpha_{m}^{*}$ and $H_{m}^{*}$ are as given in \eqref{eq:h&a}.
We carry out the test based on the chi-squared distance between the
posterior Dirichlet process $P_{m}^{*}$ and the hypothesized distribution
$F_{0}$. Note that for the large sample size, both the Pearson's
goodness-of-fit test and the likelihood ratio test (the Kullback-Leibler
distance) are asymptotically equivalent. For simplicity, we only consider
Pearson's goodness-of-fit test. Theorem \ref{prop 1} describes this
connection and the asymptotic distribution for the law of the posterior
distance for large sample size which is equivalent to the frequentist's
chi-squared test. This result follows from \citet{al2012new} and
\citet{lo1987large}, but we include a simple calculation to show
the asymptotic distribution of $D_{\alpha m}(A)=\sqrt{m}(P_{m}^{*}(A)-H_{m}^{*}(A))$
where $A\in\mathcal{X}.$ Notice that by having the partition $\{A,A^{c}\}$
and the definition of Dirichlet process,
\[
P_{m}^{*}(A)\sim Beta\left(\alpha_{m}^{*}H_{m}^{*}(A),\alpha_{m}^{*}H_{m}^{*}(A^{c})\right).
\]
Set $Y=P_{m}^{*}(A)$ and $v=H_{m}^{*}(A)$ where $P_{m}^{*}$ and
$H_{m}^{*}$ are defined in \eqref{eq:h&a}. Then, for $0<y<1,$ the
random variable $Y$ has the probability density function
\[
f(y)=\frac{\Gamma(m)}{\Gamma(mv)\Gamma(m(1-v))}y^{\alpha_{m}^{*}v-1}(1-y)^{\alpha_{m}^{*}(1-v)-1}.
\]
Thus, the probability density function of $Z=\sqrt{m}(Y-v)$ in its
support is

\begin{eqnarray}
f_{Z}(z) & = & \frac{\Gamma(m)}{\Gamma(mv)\Gamma(m(1-v))}\left(\frac{z}{\sqrt{m}}+v\right)^{\alpha_{m}^{*}v-1}\left(1-\frac{z}{\sqrt{m}}-v\right)^{\alpha_{m}^{*}(1-v)-1}.\label{eq:y}
\end{eqnarray}
By Scheffé's theorem (\citet{billingsley2013convergence}, page 29),
we need to show that
\[
f_{Z}(z)\rightarrow\frac{1}{\sqrt{2\pi\sigma^{2}}}\exp\left\{ -\frac{z^{2}}{2\sigma^{2}}\right\} ,
\]
 where $\sigma^{2}=F(A)(1-F(A))$. By Stirling\textquoteright{}s formula,
we have

\[
\Gamma\left(x\right)\thickapprox\sqrt{2\pi}x^{x-\frac{1}{2}}e^{-x}\:\textrm{as }x\rightarrow\infty,
\]
where we use the notation $f(x)\approx g(x)\,\textrm{as }x\rightarrow\infty$
if $\underset{x\rightarrow\infty}{\lim}\frac{f(x)}{g(x)}=1$. From
\eqref{eq:h&a}, as $m\rightarrow\infty$, $H_{m}^{*}\overset{a.s.}{\rightarrow}F$
and $\alpha_{m}^{*}=\alpha+m\approx m$. Then, the equation \eqref{eq:y}
can be rewritten as

\begin{eqnarray*}
f_{Z}(z) & = & \frac{\Gamma\left(m\right)}{\Gamma\left(mv\right)\Gamma\left(m\left(1-v\right)\right)}\left(\frac{z}{\sqrt{m}}+v\right)^{mv-1}\left(1-\frac{z}{\sqrt{m}}-v\right)^{m(1-v)-1},
\end{eqnarray*}
where $v=F(A)$. Then,

\begin{eqnarray}
\underset{m\rightarrow\infty}{\lim}f_{Z}(z) & = & \frac{1}{\sqrt{2\pi}}\underset{m\text{\textrightarrow\ensuremath{\infty}}}{\lim}\left\{ \frac{\left(\frac{z}{\sqrt{m}}+v\right)^{mv-1}\left(1-\frac{z}{\sqrt{m}}-v\right)^{m(1-v)-1}}{v^{mv-\frac{1}{2}}(1-v)^{m(1-v)-\frac{1}{2}}}\right\} \nonumber \\
 & = & \frac{1}{\sqrt{2\pi v(1-v)}}\underset{m\text{\textrightarrow\ensuremath{\infty}}}{\lim}\left\{ \frac{\left(\frac{z}{\sqrt{m}}+v\right)^{mv-1}\left(1-\frac{z}{\sqrt{m}}-v\right)^{m(1-v)-1}}{v^{mv-1}(1-v)^{m(1-v)-1}}\right\} \nonumber \\
 & = & \frac{1}{\sqrt{2\pi v(1-v)}}\underset{m\text{\textrightarrow\ensuremath{\infty}}}{\lim}\left\{ \left(1+\frac{z}{\sqrt{m}v}\right)^{mv-1}\left(1-\frac{z}{\sqrt{m}(1-v)}\right)^{m(1-v)-1}\right\} \nonumber \\
 & = & \frac{1}{\sqrt{2\pi v(1-v)}}\exp\left\{ \underset{m\text{\textrightarrow\ensuremath{\infty}}}{\lim}m\ln\left(\eta_{m}\right)\right\} ,\label{eq:fy}
\end{eqnarray}
where

\[
\eta_{m}=\left(1+\frac{z}{\sqrt{m}v}\right)^{v}\left(1-\frac{z}{\sqrt{m}(1-v)}\right)^{1-v}.
\]
Therefore,

\begin{eqnarray*}
\underset{m\text{\textrightarrow\ensuremath{\infty}}}{\lim}m\ln\left(\eta_{m}\right) & = & \underset{m\text{\textrightarrow\ensuremath{\infty}}}{\lim}\frac{1}{1/m}\left\{ v\ln\left(1+\frac{z}{\sqrt{m}v}\right)+(1-v)\ln\left(1-\frac{z}{\sqrt{m}(1-v)}\right)\right\} .
\end{eqnarray*}
By applying the L'Hospital's rule, we obtain

\begin{eqnarray}
\underset{m\text{\textrightarrow\ensuremath{\infty}}}{\lim}m\ln\left(\eta_{m}\right) & = & \underset{m\text{\textrightarrow\ensuremath{\infty}}}{\lim}(-m^{2})\left\{ \frac{-\frac{vz}{2vm^{3/2}}}{\left(1+\frac{z}{\sqrt{m}v}\right)}+\frac{\frac{(1-v)z}{2(1-v)m^{3/2}}}{\left(1-\frac{z}{\sqrt{m}(1-v)}\right)}\right\} \nonumber \\
 & = & \underset{m\text{\textrightarrow\ensuremath{\infty}}}{\lim}\frac{m}{2}\left\{ \frac{vz}{\sqrt{m}v+z}-\frac{(1-v)z}{\sqrt{m}(1-v)-z}\right\} \nonumber \\
 & = & \underset{m\text{\textrightarrow\ensuremath{\infty}}}{\lim}\frac{m}{2}\left\{ \frac{-z^{2}}{(\sqrt{m}v+z)(\sqrt{m}(1-v)-z)}\right\} \nonumber \\
 & = & \frac{-z^{2}}{2v(1-v)}.\label{eq:lim}
\end{eqnarray}
Substituting \eqref{eq:lim} in \eqref{eq:fy} completes the proof
of normality of $D_{\alpha m}(A)=\sqrt{m}(P_{m}^{*}(A)-H_{m}^{*}(A))$.
A similar method proves that as $m\rightarrow\infty$, for any partition
$\{A_{1},\ldots,A_{k}\}$ of the space $\mathcal{X},$

\[
(D_{\alpha m}(A_{1}),D_{\alpha m}(A_{2}),\ldots,D_{\alpha m}(A_{k}))\overset{d}{\rightarrow}(B_{F}(A_{1}),B_{F}(A_{2}),\ldots,B_{F}(A_{k})),
\]
where $B_{F}$ is the Brownian bridge.
\begin{rem}
A Gaussian process $\{B_{F}(A),A\in\mathcal{X}\}$ is called a Brownian
bridge if $E(B_{F}(A))=0$ and $Cov(B_{F}(A_{i}),B_{F}(A_{j}))=F(A_{i}\cap A_{j})-F(A_{i})\cap F(A_{j}),$
where $A_{i},A_{j}\in\mathcal{X}$. Now we can imply the following
Lemma.\end{rem}
\begin{lem}
\label{Lemm}Let $X_{1},\ldots,X_{m}$ be a random sample from the
distribution $H$. If $P_{m}^{*}$ is the Dirichlet process posterior
given $X_{1},\ldots,X_{m}$. Then, as $m\rightarrow\infty$,
\[
D_{\alpha m}(\cdot)=\sqrt{m}(P_{m}^{*}(\cdot)-H_{m}^{*}(\cdot))\overset{d}{\rightarrow}B_{F}(\cdot).
\]

\end{lem}
For a detailed proof similar to what we presented here, see \citet{al2012new}.
Also, see \citet{james2008large}, \citet{ghosal2010dirichlet} and
\citet{lo1987large}. \citet{al2012new} proved that as $\alpha\rightarrow\infty$,
$D_{\alpha}(\cdot)=\sqrt{\alpha}(P(\cdot)-H(\cdot))\overset{d}{\rightarrow}B_{H}(\cdot)$.
Theorem \ref{prop 1} describes the asymptotic distribution of the
posterior distance for a large sample size.
\begin{thm}
\label{prop 1}Suppose $X_{1},\ldots,X_{m}$ is a random sample from
a distribution $F$ on sample space $\mathcal{X}$. Let $P\thicksim DP(\alpha H)$
and $P_{m}^{*}=\left(P\mid X_{1},\ldots,X_{m}\right)\thicksim DP(\alpha_{m}^{*}H_{m}^{*})$,
where $\alpha_{m}^{*}=\alpha+m$ and $H_{m}^{*}=\frac{\alpha}{\alpha+m}H+\frac{m}{\alpha+m}\frac{\sum_{i=1}^{m}\delta_{X_{i}}}{m}$.
Let $D_{KL}(P_{m}^{*}\parallel H_{m}^{*})$ denotes the Kullback-Leibler
distance between $P_{m}^{*}$ and $H_{m}^{*}$. For any finite partition
$\{A_{1},\ldots,A_{k}\}$ of $\mathcal{X}$, define
\end{thm}
\begin{equation}
\mathcal{D}(P_{m}^{*},H_{m}^{*}):=\alpha_{m}^{*}\underset{i=1}{\overset{k}{\sum}}\frac{(P_{m}^{*}(A_{i})-H_{m}^{*}(A_{i}))^{2}}{H_{m}^{*}(A_{i})}.\label{eq:chi1}
\end{equation}
Then, as $m\rightarrow\infty,$ we have

\[
2\alpha_{m}^{*}D_{KL}(P_{m}^{*}\parallel H_{m}^{*})\simeq\mathcal{D}(P_{m}^{*},H_{m}^{*})\overset{d}{\rightarrow}\chi_{(k-1)}^{2}.
\]

\begin{proof}
See the Appendix.
\end{proof}
Note that as the sample size $m$ increases, $H_{m}^{*}\overset{a.s.}{\rightarrow}F$
and therefore the posterior Dirichlet process $P_{m}^{*}$ converges
to the true underlying distribution $F$ of the observed data $X_{1},\ldots,X_{m}$.
In our methodology, we compute the observed probability for bin $A_{i},\, i=1,\ldots,k$
of the partition $\{A_{1},\ldots,A_{k}\}$ by calculating the posterior
probability $P_{m}^{*}(A_{i}),\, i=1,\ldots,k$. Notice that in our
Bayesian paradigm, we need to embed our prior information in our test
statistic. In other words, the base distribution and the concentration
parameter plays the role of the prior knowledge. Moreover, we do not
count the observed frequencies in each bin. Instead, we calculate
the exact posterior probability for each bin. Then, the $X^{2}$ distance
in \eqref{eq:chi1} compares the posterior probabilities with the
hypothesized ones. Additionally, there is no need to apply the asymptotic
distribution as we know the exact distribution of the $X^{2}$ distance
via a Monte Carlo simulation. Also, There are many discussions for
choosing the number of bins in the literature and different criterion
are suggested by various authors. See, for example, \citet{koehler1990chi},
\citet{mann1942choice}, \citet{williams1950choice}, \citet{watson1957chi2},
\citet{hamdan1963number}, \citet{dahiya1973many}, \citet{gvanceladze1979tests},
\citet{best1981two}, \citet{quine1985efficiencies} and \citet{johnson2004bayesian}.
In the following subsections, we first use the distance \eqref{eq:chi1}
to find an appropriate concentration parameter for the Dirichlet process.
Then, we carry out a Bayesian nonparametric chi-squared goodness-of-fit
test. We also extend our method to present a Bayesian nonparametric
test of independence. The described methods will be illustrated by
some examples in Section 6.

\subsection{Selection of the concentration parameter of Dirichlet process}

A challenging question in Bayesian nonparametric is to determine $\alpha$,
the concentration parameter of the prior. To suggest an appropriate
concentration parameter $\alpha$, fix $c$ and $q$ such that

\begin{equation}
Pr(\mathcal{D}(P,F_{0})\leq c)=q,\label{eq:prior-p}
\end{equation}
where
\[
\mathcal{D}=\mathcal{D}(P,F_{0})=\alpha\underset{i=1}{\overset{k}{\sum}}\frac{(P(A_{i})-F_{0}(A_{i}))^{2}}{F_{0}(A_{i})}.
\]
Throughout this paper, $\mathcal{D}=\mathcal{D}(P,F_{0})$ denotes
the prior distance. Also, let $\mathcal{D}^{*}=\mathcal{D}(P_{m}^{*},F_{0})$
stands for the posterior distance as given in \eqref{eq:chi1}, replacing
$H_{m}^{*}$ by $F_{0}$. We can approximate the distribution of the
prior distance $\mathcal{D}=\mathcal{D}(P,F_{0})$ by the empirical
distribution of $N$ randomly generated values from $\mathcal{D}$.
Thus, \eqref{eq:prior-p} can be approximated by the proportion of
$\mathcal{D}$ values that are less than or equal to $c$. We start
with an initial value of $\alpha$ and then we compute the probability
\eqref{eq:prior-p}. If the probability is close to the value of $q$,
we choose $\alpha$, otherwise, we repeat this procedure by increasing
or decreasing the value of $\alpha$ to reach the value of $q.$ The
results of a simulation study for an illustrated example are summarized
in Table \ref{tab:1} in Section 6.

\subsection{Goodness-of-fit test }

Suppose $X_{1},\ldots,X_{m}$ is a random sample from a distribution
$F$. In order to test the null hypothesis $H_{0}:F=F_{0}$, we place
the Dirichlet process prior with parameters $\alpha$ and $F_{0}$
on $F$. Then, since under the null hypothesis, the true distribution
of data is $F_{0}$, we calculate the distance between the Dirichlet
process prior and $F_{0}$. The appropriate concentration parameter
$\alpha$ of the Dirichlet process can be calculated by the method
explained in Subsection 4.1. We follow the approach of \citet{swartz1999nonparametric}.
That is, for a fixed value of $q$ and $c$, we obtain $\alpha$ by
\eqref{eq:prior-p}. Having $\alpha$, we generate a random sample
of size $N$ from the Dirichlet process posterior with parameters
$\alpha_{m}^{*}$ and $H_{m}^{*}$ as given earlier to get $N$ random
samples of $\mathcal{D}^{*}=\mathcal{D}(P_{m}^{*},F_{0})$ as given
in Theorem \ref{prop 1}. The distribution of $\mathcal{D}^{*}$ can
be estimated by the empirical distribution of $\mathcal{D}^{*}$ values.
Hence, the posterior probability $Pr(\mathcal{D}(P_{m}^{*},F_{0})\leq c)$
can be estimated by the proportion of $\mathcal{D}^{*}$ which are
less than or equal to $c$. Here, our decision making is based on
the comparison of the posterior probability and the prior probability
$q$, where $q$ represents the prior belief that the underlying distribution
$F$ is practically equivalent to $F_{0}$. Usually $q=0.5$ is considered.
If the empirical posterior probability $Pr(\mathcal{D}(P_{m}^{*},F_{0})\leq c)$
is less than $q$, we reject the null hypothesis, otherwise there
is no evidence to reject the null hypothesis.

Similar to the frequentist's chi-squared goodness-of-fit test, we
can also generalized the test to a family of distributions. Now, consider
the null hypothesis $H_{0}:F=F_{\theta}$ for some $\theta\in\Theta.$
Therefore, the true underlying distribution $F$ is a member of a
family of distributions indexed by the parameter $\theta.$ Our approach
for this case is similar to the simple hypothesis with the addition
of a prior distribution $\pi(\theta)$ on $\theta$. Thus, the distance
$\mathcal{D}(P_{m}^{*},F_{\theta})$ depends on the unknown parameter
$\theta$. In order to conduct the test, we first generate a random
sample from the posterior distribution of $\theta$ given $X_{1},\ldots,X_{m}$
that is given as

\begin{equation}
g(\theta\mid X_{1},\ldots,X_{m})\propto\left(\overset{m}{\underset{i=1}{\prod}}f_{\theta}(x_{i})\right)\pi(\theta),\label{eq:post}
\end{equation}
where $f_{\theta}(x)$ is the density function corresponding to $F_{\theta}$.
By having a specified $c$ and $q$, we find the parameter $\alpha$
such that $Pr(\mathcal{D}(P,F_{\widehat{\theta}})\leq c)=q,$ where
$\widehat{\theta}=E(\theta)$. Then, we generate a random sample $\theta_{i}^{*},\, i=1,\ldots,M$
from the posterior distribution $g(\theta\mid X_{1},\ldots,X_{m})$.
We obtain $\theta_{Min}=\arg\underset{\theta_{i}^{*}}{\min}D(P_{m}^{*},F_{\theta_{i}^{*}}),\, i=1,\ldots,M$,
where $P_{m}^{*}$ is the posterior Dirichlet process with the base
distribution $H_{\theta_{i}^{*}}^{*}$ as given in \eqref{eq:h&a}
with $H$ replaced by $H_{\theta_{i}^{*}}$. We then generate a sample
of size $N$ from $\mathcal{D}(P_{m}^{*},F_{\theta_{Min}})$. Similar
to the case of testing for the simple hypothesis, the decision is
made by comparing the posterior probability $Pr(\mathcal{D}(P_{m}^{*},F_{\theta_{Min}})\leq c)$
and $q$. Note that in the case of a non-standard distribution in
\eqref{eq:post}, in order to sample from the posterior distribution,
we need to apply some specialized techniques such as Metropolis-Hastings
algorithm. In Section 6, some examples with simulation study are illustrated
for the simple hypothesis $H_{0}:F=N(0,1)$ and the null hypothesis
$H_{0}:F=\exp(\theta)$ with a Gamma (1.7, 2550) prior distribution
for $\theta$.

\section{Bayesian nonparametric chi-squared test of independence}

Here, we describe a Bayesian nonparametric chi-squared test of independence
of two random variables. The null hypothesis of the chi-squared test
of independence is given as $H_{0}:F_{X,Y}(x,y)=F_{X}(x)F_{Y}(y)$
against the alternative $H_{0}:F_{X,Y}(x,y)\neq F_{X}(x)F_{Y}(y)$
and hence it examines whether there is a significant relationship
between two random variables $X$ and $Y$. Suppose $\{A_{j}\}_{j=1,\ldots,r}$
is a partition of the space $\mathcal{X}$ of the random variable
$X$ and $\{B_{k}\}_{k=1,\ldots,s}$ is a partition of the space $\mathcal{Y}$
of the random variable $Y$, i.e., $\mathcal{X}=\underset{j=1}{\overset{r}{\cup}}A_{j}$
and $\mathcal{Y}=\underset{k=1}{\overset{s}{\cup}}B_{k}$. Let $(X_{l},Y_{l})\overset{i.i.d}{\thicksim}F(x,y),\, l=1,\ldots,m$
be the sample data and $H$ be a bivariate distribution. Then, the
Dirichlet process posterior with parameters $H_{m}^{*}$ and $\alpha_{m}^{*}$
is written as $P_{m}^{*}=\underset{i=1}{\overset{\infty}{\sum}}p_{i}^{\left(m\right)}\delta_{(X_{i}^{*},Y_{i}^{*})}$,
where $p_{i}^{(m)}$ is as given in \eqref{eq:pFer}, $\alpha$ is
replaced by $\alpha_{m}^{*}$ and $(X_{i}^{*},Y_{i}^{*}),\, i=1,\ldots,n$
are generated from $H_{m}^{*}=\frac{\alpha}{\alpha+m}H+\frac{m}{\alpha+m}\frac{\sum_{i=1}^{m}\delta_{(X_{i},Y_{i})}}{m}$.
In our new approach, we compute the observed probability at level
$j$ of the random variable $X$ and at level $k$ of the random variable
$Y$ by $P_{m}^{*}(A_{j}\times B_{k})$ and the corresponding expected
probability is computed as $P_{m}^{*}(A_{j}\times\mathcal{Y})P_{m}^{*}(\mathcal{X}\times B_{k})$,
where
\begin{equation}
P_{m}^{*}(A_{j}\times B_{k})=\underset{i=1}{\overset{\infty}{\sum}}p_{i}^{(m)}\delta_{(X_{i}^{*},Y_{i}^{*})}(A_{j}\times B_{k})\label{eq:p1}
\end{equation}
and
\begin{eqnarray}
 &  & P_{m}^{*}(A_{j}\times\mathcal{Y})=\underset{i=1}{\overset{\infty}{\sum}}p_{i}^{(m)}\delta_{(X_{i}^{*},Y_{i}^{*})}(A_{j}\times\mathcal{Y})=\underset{i=1}{\overset{\infty}{\sum}}p_{i}^{(m)}\delta_{X_{i}^{*}}(A_{j})\nonumber \\
 &  & P_{m}^{*}(\mathcal{X}\times B_{k})=\underset{i=1}{\overset{\infty}{\sum}}p_{i}^{(m)}\delta_{(X_{i}^{*},Y_{i}^{*})}(\mathcal{X}\times B_{k})=\underset{i=1}{\overset{\infty}{\sum}}p_{i}^{(m)}\delta_{Y_{i}^{*}}(B_{k}).\label{eq:p2}
\end{eqnarray}
 Then, test statistic is given as
\begin{equation}
\mathcal{D}^{*}=\alpha_{m}^{*}\underset{k=1}{\overset{s}{\sum}}\underset{j=1}{\overset{r}{\sum}}\frac{(P_{m}^{*}(A_{j}\times B_{k})-P_{m}^{*}(A_{j}\times\mathcal{Y})P_{m}^{*}(\mathcal{X}\times B_{k}))^{2}}{P_{m}^{*}(A_{j}\times\mathcal{Y})P_{m}^{*}(\mathcal{X}\times B_{k})}\label{eq:d2-1}
\end{equation}
which asymptotically converges to  $\chi_{(r-1)\times(s-1)}^{2}$.
In order to carry out the test, we proceed a similar process as explained
in Section 4 for the goodness-of-fit test. We generate a random sample
of size $N$ from the prior distance $\mathcal{D}$, where $\mathcal{D}$
is computed by \eqref{eq:d2-1} replacing $\alpha_{m}^{*}$ by $\alpha$
and the Dirichlet process posterior $P_{m}^{*}$ by the Dirichlet
process prior $P$. By having a fixed value $c$ and a fixed probability
$q$, an appropriate concentration parameter $\alpha$ is obtained
by the equation $Pr(\mathcal{D}\leq c)=q$. Then, by generating a
sample of size $N$ from $\mathcal{D}^{*}$, we can approximate the
distribution of $\mathcal{D}^{*}$ by the empirical distribution of
$\mathcal{D}^{*}$ values. Our decision is made by comparing the probabilities
$Pr(\mathcal{D}^{*}\leq c)$ and $q$ and we reject the null hypothesis
if $Pr(\mathcal{D}^{*}\leq c)$ is less than $q$. An illustrative
example with a simulation study is discussed in Section 6.

\section{Simulation study}

This section provides some examples with simulation studies for the
Bayesian nonparametric tests described in Section 4 and 5. For all
the simulations, we use the finite sum representation to approximate
the Dirichlet process as given in \eqref{eq:pd}.
\begin{example}
We consider a Dirichlet process with the base distribution $H=N(0,1)$
and $n=2000$ terms in the finite sum representation \eqref{eq:pd}.
We partition the space into $k=7$ bins. Table \ref{tab:1} represents
the probability \eqref{eq:prior-p} when $F_{0}=N(0,1).$ The probabilities
are computed for various values of $\alpha$ and $c$ and for a simulation
of size $N=2000.$ As the Table \ref{tab:1} shows, for example, if
we set $q=0.48$ and $c=3$, $\alpha=10$ is an appropriate concentration
parameter.
\end{example}
\noindent
\begin{table}[H]
\begin{centering}
\begin{tabular}{|c|c|c|c|c|c|c|}
\cline{2-7}
\multicolumn{1}{c|}{} & \multicolumn{6}{c|}{$Pr(\mathcal{D}(P,F_{0})<c)$}\tabularnewline
\hline
$\alpha$ & $c=1$ & $c=2$ & $c=3$ & $c=4$ & $c=5$ & $c=6$\tabularnewline
\hline
1 & 0.298 & 0.745 & 0.812 & 0.857 & 0.893 & 0.933\tabularnewline
\hline
10 & 0.068 & 0.273 & 0.480 & 0.624 & 0.717 & 0.781\tabularnewline
\hline
50 & 0.029 & 0.143 & 0.311 &  0.474 & 0.612 & 0.696\tabularnewline
\hline
100 & 0.027 & 0.116 & 0.258 &  0.409 & 0.540 & 0.648\tabularnewline
\hline
200 & 0.020 & 0.094 & 0.219 & 0.353 & 0.492 & 0.595\tabularnewline
\hline
300 & 0.011 & 0.073 & 0.179 & 0.297 & 0.432 & 0.542\tabularnewline
\hline
500 & 0.009 & 0.057 & 0.150 & 0.263 & 0.368 & 0.484\tabularnewline
\hline
\end{tabular}
\par\end{centering}

\caption{\label{tab:1}{\small{The computed the probability $Pr\left(\mathcal{D}(P,F_{0})<c\right)$
for different choices of $\alpha$ and $c$ in Example 6.1.}}}
\end{table}

\begin{example}
Suppose $X_{1},\ldots,X_{150}$ is a random sample from a standard
Cauchy distribution. We want to test the null hypothesis $H_{0}:F=N(0,1)$.
We divide the sample space into $k=7$ bins $A_{i},\, i=1,\ldots,7$
as given in Table \ref{tab:A-sample-table-1} and $Pr(A_{i})$ shows
the observed probability of each bin. . We consider $H=N(0,1)$ as
the base measure and $n=2000$ terms in the finite sum representation
of Dirichlet process as given in \eqref{eq:pd}. Then, an appropriate
concentration parameter $\alpha=100$ is obtained when $q=0.54$ and
$c=5$. By sampling $N=2000$ times from the Dirichlet process posterior
$P_{m}^{*}$ and then $N=2000$ realizations of $\mathcal{D}^{*}$,
we obtain $Pr(\mathcal{D}(P_{m}^{*},F_{0})\leq c)=0$. Thus, we reject
the normality hypothesis of the data. Our decision is consistent with
the classical chi-squared test which gives a p-value of $2.2\times10^{-16}$.
Also, our decision is consistent with other choices of the base measure
$H$, since the Dirichlet process posterior converges to the true
underlying distribution as the data size increases. Table \ref{tab:A-sample-table-1}
illustrates the observed probabilities obtained by counting the data
points in each bin and the corresponding probabilities computed by
the Dirichlet process posterior.

\begin{table}[H]
\begin{centering}
{\scriptsize{}}%
\begin{tabular}{|c|c|c|c|c|c|c|c|}
\cline{2-8}
\multicolumn{1}{c|}{} & \multicolumn{7}{c|}{{\scriptsize{$X$}}}\tabularnewline
\cline{2-8}
\multicolumn{1}{c|}{} & {\scriptsize{$A_{1}=(-\infty,-2]$}} & {\scriptsize{$A_{1}=(-2,-1]$}} & {\scriptsize{$A_{2}=(-1,0]$}} & {\scriptsize{$A_{3}=(0,1]$}} & {\scriptsize{$A_{4}=(1,2]$}} & {\scriptsize{$A_{5}=(2,3]$}} & {\scriptsize{$A_{5}=(3,\infty)$}}\tabularnewline
\hline
{\scriptsize{$Pr(A_{i})$}} & {\scriptsize{0.133 }} & {\scriptsize{0.100}} & {\scriptsize{0.313}} & {\scriptsize{0.240}} & {\scriptsize{0.060}} & {\scriptsize{0.067}} & {\scriptsize{0.087}}\tabularnewline
\hline
{\scriptsize{$P_{m}^{*}(A_{i})$}} & {\scriptsize{0.072 }} & {\scriptsize{0.131}} & {\scriptsize{0.342}} & {\scriptsize{0.310}} & {\scriptsize{0.069}} & {\scriptsize{0.030}} & {\scriptsize{0.046}}\tabularnewline
\hline
{\scriptsize{$F_{0}(A_{i})$}} & {\scriptsize{0.023}} & {\scriptsize{0.136}} & {\scriptsize{0.341 }} & {\scriptsize{0.341}} & {\scriptsize{0.136}} & {\scriptsize{0.022}} & {\scriptsize{0.001}}\tabularnewline
\hline
\end{tabular}
\par\end{centering}{\scriptsize \par}

\caption{{\small{\label{tab:A-sample-table-1}The computed probabilities $Pr(A_{i}),$
}}$P_{m}^{*}(A_{i})${\small{ and $F_{0}(A_{i})$ where $Pr(A_{i})$
is the observed probability obtained by counting the data points in
$i$th bin, }}$P_{m}^{*}(A_{i})${\small{ is the corresponding probability
computed by the Dirichlet process posterior for one simulation and
$F_{0}(A_{i})$ shows the corresponding expected probability under
the null hypothesis}}.}
\end{table}

Figure \ref{fig:Q-Q-1} shows the Q-Q plot, the empirical distribution
and the histogram of $N=2000$ randomly generated from the prior distance
$\mathcal{D}=\mathcal{D}(P,F_{0})$ compared with a $\chi_{(4)}^{2}$
distribution, respectively.
\end{example}
$\,$
\begin{example}
(Example 3.6. \citet{hamada2008bayesian}) Suppose we have an observed
data of size $m=31$ for the lifetime of the liquid crystal display
(LCD) projector lamps. We want to test if the lifetime distribution
of the liquid crystal display (LCD) projector lamps is an Exponential
distribution with parameter $\theta>0$. That is, we want to test
the null hypothesis $H_{0}:F_{\theta}=Exp(\theta)$, where $\theta$
has a Gamma (1.7, 2550) prior distribution. Hence, the posterior distribution
of $\theta$ given data is a Gamma (32.7, 20457) distribution. We
consider $k=4$ bins. By specifying the values $q=0.51$ and $c=3$,
the appropriate $\alpha=100$ is obtained. We obtain $\theta_{1}^{*},\ldots,\theta_{M}^{*}$
as realizations from the distribution of $\left(\theta\mid X_{1},\ldots,X_{31}\right)$
and we get $\theta_{Min}=0.00136$. By generating $N=2000$ times
from $\mathcal{D}^{*}=\mathcal{D}(P_{m}^{*},F_{\theta_{Min}})$, we
obtain $Pr(\mathcal{D}(P_{m}^{*},F_{\theta_{Min}})\leq c)=0.71$.
Hence, there is no evidence to reject the null hypothesis.
\end{example}
\,
\begin{example}
Suppose we have a random sample $(X_{i},Y_{i}),\, i=1,\ldots,150$
from a bivariate normal distribution $F=N_{2}(\mathbf{\boldsymbol{\mu}},\Sigma)$
where $\boldsymbol{\mu}=\left[\begin{array}{c}
0\\
0
\end{array}\right]$ and $\Sigma=\left[\begin{array}{cc}
10 & 3\\
3 & 2
\end{array}\right].$ We consider five levels of variable $X$ and four levels of variable
$Y$ as given in Table \ref{tab:indep}. We want to test the null
hypothesis of independence as given in Section 5. Consider a Dirichlet
process prior with base distribution  $H=N(\boldsymbol{\mu}_{1},\Sigma_{1})$
where $\mathbf{\boldsymbol{\mu}}_{1}=\left[\begin{array}{c}
0\\
0
\end{array}\right]$ and $\Sigma_{1}=\left[\begin{array}{cc}
1 & 0\\
0 & 1
\end{array}\right].$ For $q=0.5$ and $c=20$, by generating $N=2000$ times from $\mathcal{D}$
and solving the equation $Pr(\mathcal{D}<c)=q$, we obtain an appropriate
concentration parameter $\alpha=100$. By generating a sample of size
$N=2000$ from the posterior distance $\mathcal{D}^{*},$ we have
$Pr(\mathcal{D}^{*}<c)=0$. Therefore, we reject the null hypothesis
of independence. The p-value of $8.34\times10^{-6}$ obtained by the
classical chi-squared test of independence results in the same conclusion.
Table {\footnotesize{\ref{tab:indep} }}represents the probability
of each category calculated by the Dirichlet process posterior.

\begin{table}[H]
\begin{centering}
{\footnotesize{}}%
\begin{tabular}{|c|c|c|c|c|c|c|}
\cline{3-7}
\multicolumn{1}{c}{} &  & \multicolumn{5}{c|}{{\footnotesize{$X$}}}\tabularnewline
\cline{3-7}
\multicolumn{1}{c}{} &  & {\footnotesize{$A_{1}=(-\infty,-1]$}} & {\footnotesize{$A_{2}=(-1,0]$}} & {\footnotesize{$A_{3}=(0,1]$}} & {\footnotesize{$A_{4}=(1,2]$}} & {\footnotesize{$A_{5}=(2,\infty)$}}\tabularnewline
\hline
\multirow{4}{*}{{\footnotesize{$Y$}}} & {\footnotesize{$B_{1}=(-\infty,-1]$}} & {\footnotesize{0.076}} & {\footnotesize{0.069}} & {\footnotesize{0.005}} & {\footnotesize{0.066}} & {\footnotesize{0.000}}\tabularnewline
\cline{2-7}
 & {\footnotesize{$B_{2}=(-1,0]$}} & {\footnotesize{0.075}} & {\footnotesize{0.031}} & {\footnotesize{0.086}} & {\footnotesize{0.063}} & {\footnotesize{0.006}}\tabularnewline
\cline{2-7}
 & {\footnotesize{$B_{3}=(0,1]$}} & {\footnotesize{0.072}} & {\footnotesize{0.047}} & {\footnotesize{ 0.045}} & {\footnotesize{0.048}} & {\footnotesize{0.043}}\tabularnewline
\cline{2-7}
 & {\footnotesize{$B_{4}=(1,\infty)$}} & {\footnotesize{0.014}} & {\footnotesize{ 0.061}} & {\footnotesize{0.044}} & {\footnotesize{0.025}} & {\footnotesize{0.125}}\tabularnewline
\hline
\end{tabular}
\par\end{centering}{\footnotesize \par}

\caption{{\small{\label{tab:indep}A sample table of probabilities computed
by the Dirichlet process posterior in Example 6.4.}}}
\end{table}

\end{example}

\section{\noindent Discussion}

In this paper, we proposed a Bayesian nonparametric chi-squared goodness
of fit test based on the Kullback-Leibler distance between the Dirichlet
process posterior and the hypothesized distribution. Our method proceeds
by placing a Dirichlet process prior on the distribution of observed
data and computing the probability of each bin of the partition from
the Dirichlet process posterior. The suggested method is in contrast
with the frequentist's Pearson's chi-squared goodness of fit test
which is based on counting the observations in each bin of the partition.
We also extended our method to present a Bayesian nonparametric test
of independence. Like the classical chi-squared test, we can generalize
our goodness-of-fit test to several variables. For categorical observations
with finite many categories, placing a Dirichlet distribution prior
on the probabilities of categories and deriving the posterior Dirichlet
distribution can establish similar tests. For example, the test of
independence and conditional independence of qualitative observations
follow easily.

\section*{Acknowledgments}

This research was supported by grant funds from the Natural Science
and Engineering Research Council of Canada.

\section*{\textmd{\normalsize{\bibliographystyle{elsart-harv}
\bibliography{BIB}
 }}}

\section*{Appendix - Proofs of Theoretical Results}

Proof of Lemma \ref{Lemma}:

Suppose that the sample space is partitioned as $x_{(1)}<\cdots<x_{(n+1)}$
such that $x_{(i)}<\theta_{(i)}<x_{(i+1)},$$\, i=1,\ldots,n$. By
definition of the Kullback-Leibler distance, we have

\begin{flushleft}
\begin{eqnarray}
D_{KL}(P_{n} & \parallel & F)=\overset{n}{\underset{i=1}{\sum}}\triangle P_{n}(x_{i})\log\left(\frac{\triangle P_{n}(x_{i})}{\triangle F(x_{i})/\triangle x_{i}}\right)\nonumber \\
 &  & =\overset{n}{\underset{i=1}{\sum}}\triangle P_{n}(x_{i})\log(\triangle P_{n}(x_{i}))-\overset{n}{\underset{i=1}{\sum}}\triangle P_{n}(x_{i})\log\left(\frac{\triangle F(x_{i})}{\triangle x_{i}}\right)\nonumber \\
 &  & =\overset{n}{\underset{i=1}{\sum}}p_{i,n}\log(p_{i,n})-\overset{n}{\underset{i=1}{\sum}}p_{i,n}\log\left(\frac{\triangle F(x_{i})}{\triangle x_{i}}\right)\nonumber \\
 &  & =-\mathcal{H}(\mathbf{p})-\overset{n}{\underset{i=1}{\sum}}p_{i,n}\log\left(\frac{\triangle F(x_{i})}{\triangle x_{i}}\right),\label{eq:dd1}
\end{eqnarray}
where $\triangle F(x_{i})=F(x_{(i+1)})-F(x_{(i)}),$$\,\triangle x_{i}=x_{(i+1)}-x_{(i)}$,
$p_{i,n}=P_{n}(x_{(i+1)})-P_{n}(x_{(i)})=P_{n}(\theta_{(i)})$ and
$\mathcal{H}(\mathbf{p})=-\overset{n}{\underset{i=1}{\sum}}p_{i,n}\log(p_{i,n})$
is the entropy of $P_{n}$. Similarly, we get
\par\end{flushleft}

\begin{eqnarray}
D_{KL}(F & \parallel & P_{n})=-\mathcal{H}(\mathbf{q})-\overset{n}{\underset{i=1}{\sum}}q_{i}\log(p_{i,n}),
\end{eqnarray}
where $q_{i}=\frac{\triangle F(x_{i})}{\triangle x_{i}}$ and $\mathcal{H}(\mathbf{q})=-\overset{n}{\underset{i=1}{\sum}}q_{i}\log q_{i}.$

Proof of Theorem \ref{prob d} and Remark \ref{remark d}:

We have $(p_{1,n},\ldots,p_{n,n})\thicksim\textrm{Dir}(\alpha/n,\ldots,\alpha/n)$.
Thus, $p_{i,n}\thicksim\textrm{Beta}(\frac{\alpha}{n},\alpha(1-\frac{1}{n})),\, i=1,\ldots,n$
and all computations for the mean and variance simply follow.

Proof of Theorem \ref{prop 1}:

We basically mimic the proof for the asymptotic frequentist's chi-squared
goodness-of-fit test. Define

\begin{equation}
\mathcal{D}^{*}=(\alpha+m)\underset{i=1}{\overset{k}{\sum}}\frac{(P_{m}^{*}(A_{i})-H_{m}^{*}(A_{i}))^{2}}{H_{m}^{*}(A_{i})}.\label{eq:a1}
\end{equation}
Let $\mathbf{Y}_{m}^{T}=(Y_{1,m},\ldots,Y_{k,m})=(P_{m}^{*}(A_{1}),\ldots,P_{m}^{*}(A_{k}))$
and\\
 $\mathbf{v}_{m}^{T}=(v_{1,m},\ldots,v{}_{k,m})=(H_{m}^{*}(A_{1}),\ldots,H_{m}^{*}(A_{k}))$.
By Lemma \ref{Lemm}, as $m\rightarrow\infty$,

\begin{equation}
\sqrt{\alpha+m}(\mathbf{Y}_{m}-\mathbf{v}_{m})^{T}\overset{d}{\rightarrow}N_{k}(\mathbf{0},\Sigma).\label{eq:y1}
\end{equation}
In here, $\Sigma=\left(\sigma_{ij}\right)_{k\times k}$ is the covariance
matrix with $\sigma_{ii}^{2}=\textrm{var}(Y_{i,m})=F(A_{i})(1-F(A_{i})),\, i=1,\ldots,k$
and $\sigma_{ij}=cov(Y_{i,m},Y_{j,m})=-F(A_{i})F(A_{j})$. Then, \eqref{eq:a1}
can be written as

\begin{equation}
\mathcal{D}^{*}=(\alpha+m)(\mathbf{Y}_{m}-\mathbf{v}_{m})^{T}\Sigma{}^{-1}(\mathbf{Y}_{m}-\mathbf{v}_{m}).\label{eq:D}
\end{equation}
 Note that the sum of the $j$th column of $\Sigma$ is $F(A_{j})\lyxmathsym{\textminus}F(A_{j})(F(A_{1})+\cdots+F(A_{k}))=0$,
that implies the sum of the rows of $\Sigma$ is the zero vector,
therefore $\Sigma$ is not invertible. To avoid dealing with this
singular matrix, we define $\mathbf{Y}_{m}^{*T}=(Y_{1,m},\ldots,Y_{k-1,m})$.
Let $\mathbf{Y}_{m}^{*}$ be the vector consisting of the first $k-1$
components of $\mathbf{Y}_{m}.$ Then, the covariance matrix of $\mathbf{Y}_{m}^{*}$
is the upper-left $(k-1)\times(k-1)$ sub-matrix of $\Sigma$ which
is denoted by $\Sigma^{*}$. Similarly, let $\mathbf{v}_{m}^{*T}$
denotes the vector $\mathbf{v}_{m}^{*T}=(v_{1,m},\ldots,v{}_{k-1,m})$.
It can be verified simply that $\Sigma^{*}$ is invertible. Furthermore,
\eqref{eq:D} can be rewritten as

\begin{equation}
\mathcal{D}^{*}=(\alpha+m)(\mathbf{Y}_{m}^{*}-\mathbf{v}_{m}^{*})^{T}(\Sigma^{*})^{-1}(\mathbf{Y}_{m}^{*}-\mathbf{v}_{m}^{*}).\label{eq:D-1}
\end{equation}
Define

\[
\mathbf{Z}_{m}^{T}=\sqrt{\alpha+m}(\Sigma^{*})^{-1/2}(\mathbf{Y}_{m}^{*}-\mathbf{v}_{m}^{*})^{T}.
\]
The central limit theorem implies $\mathbf{Z}_{m}^{T}\overset{d}{\rightarrow}N_{k-1}(\mathbf{0},I)$.
By definition, the $\chi_{(k-1)}^{2}$ distribution is the distribution
of the sum of the squares of $k\lyxmathsym{\textminus}1$ independent
standard normal random variables. Therefore,

\[
\mathcal{D}^{*}=\mathbf{Z}_{m}^{T}\mathbf{Z}_{m}\overset{d}{\rightarrow}\chi_{(k-1)}^{2}.
\]
\newpage{}

\noindent
\begin{figure}[H]
\centering{}\includegraphics[scale=0.35]{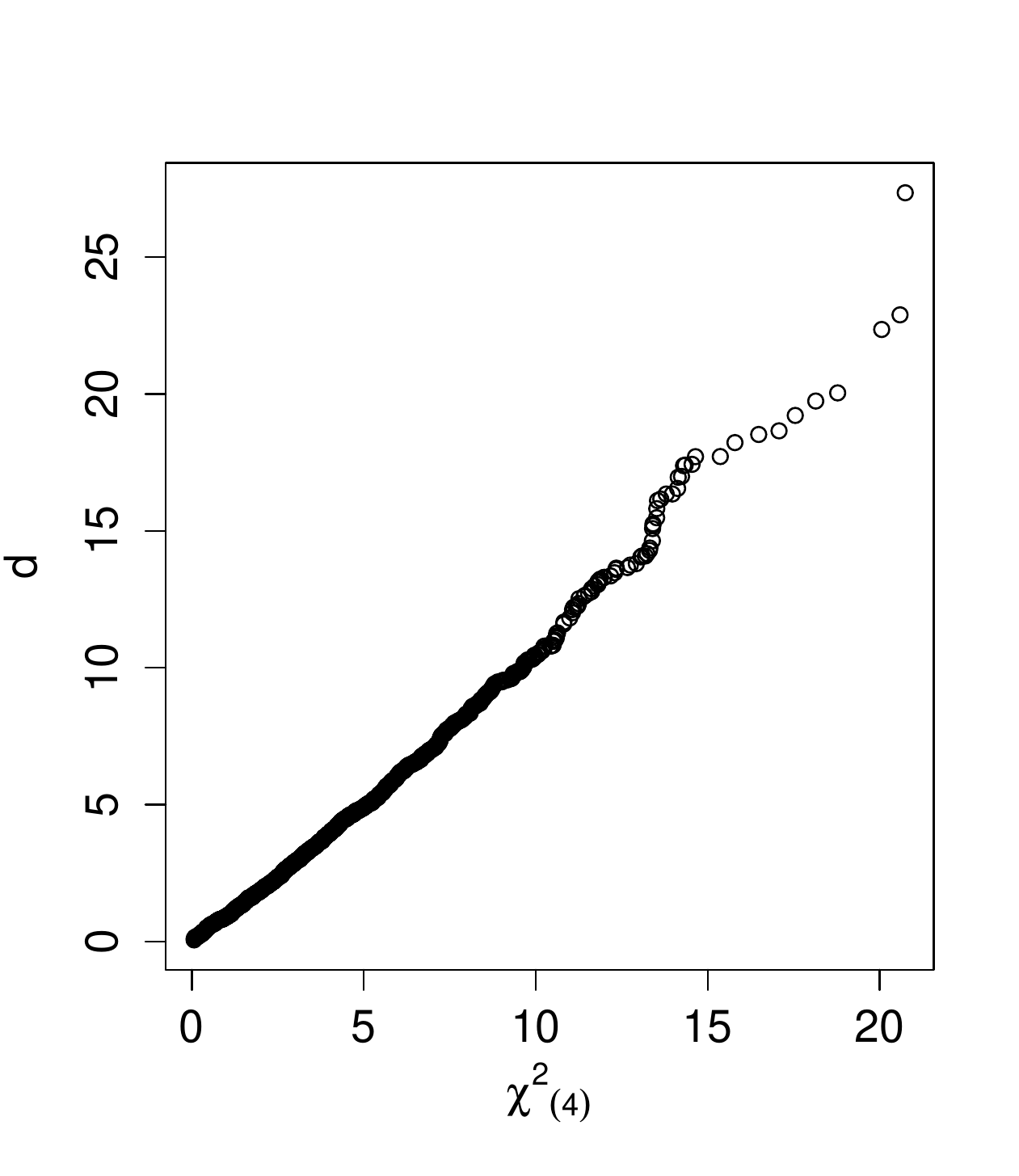}\includegraphics[scale=0.35]{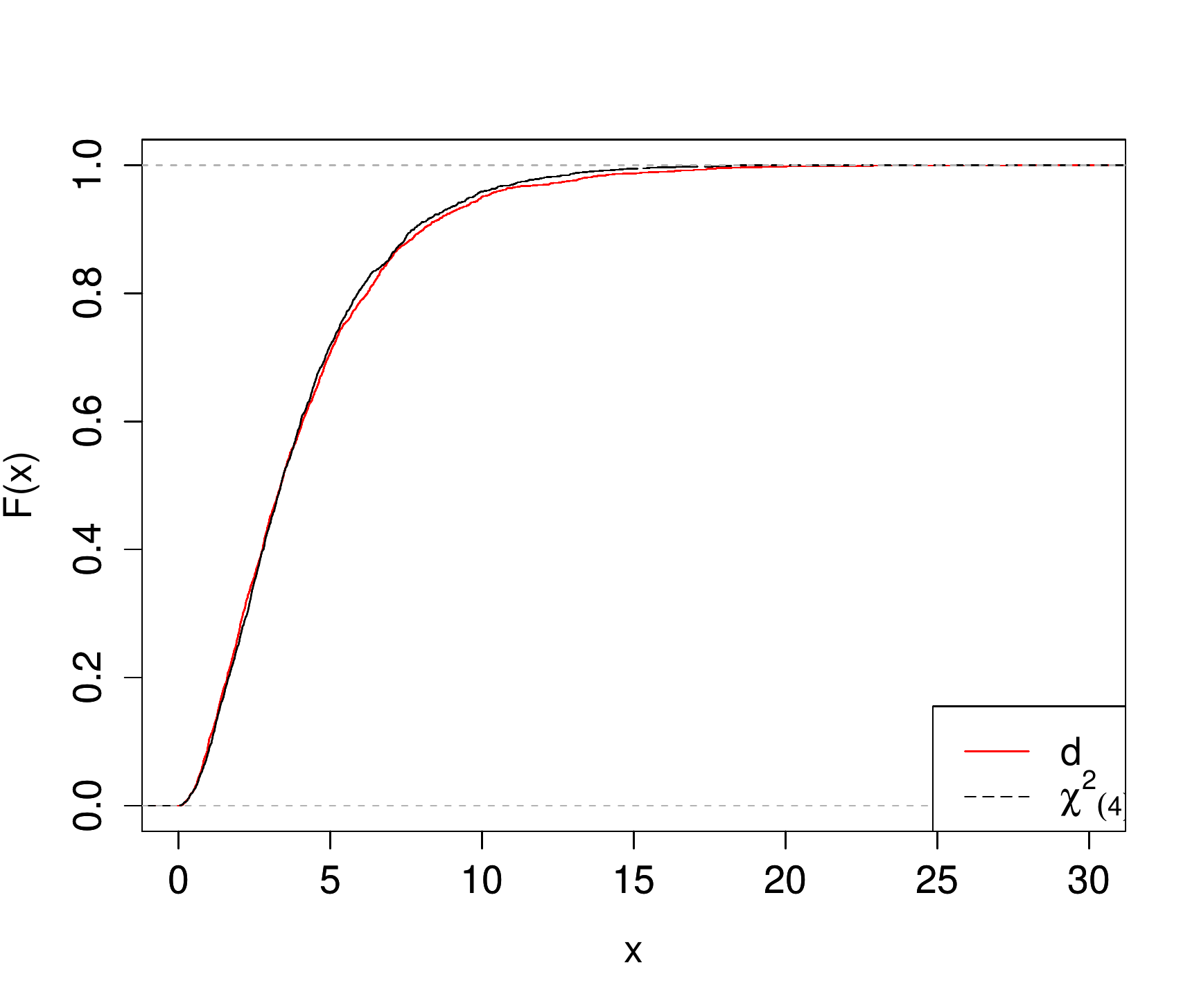}\includegraphics[scale=0.35]{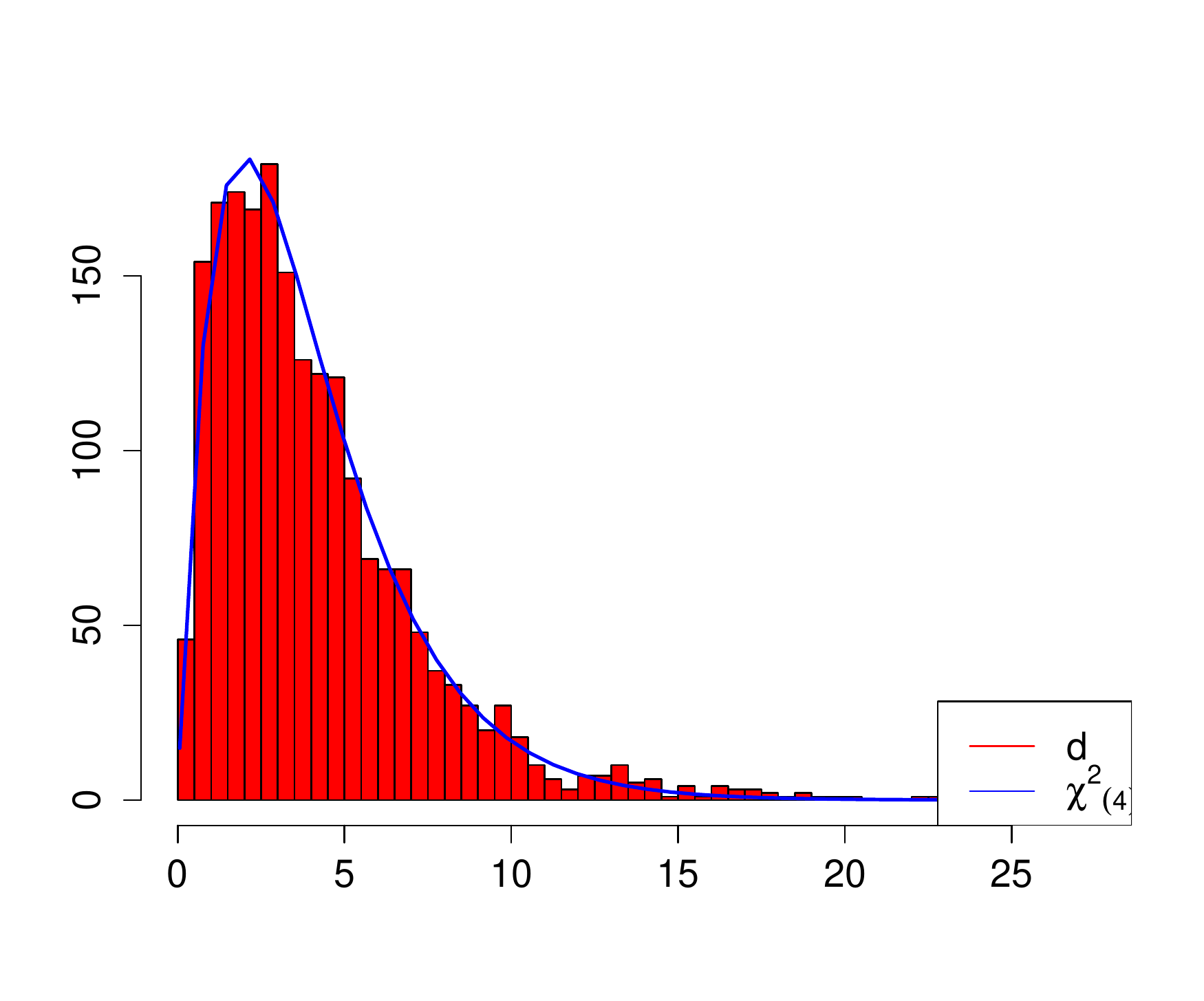}\caption{{\small{\label{fig:Q-Q-1}}} (Left) The Q-Q plot of $N=2000$ realizations
of $\mathcal{D}=\mathcal{D}(P,H)$ with $\alpha=100$, $H=N(0,1)$,
$k=5$ and $n=3000$ versus a $\chi_{(4)}^{2}$ distribution. (Middle)
The empirical distribution function of $\mathcal{D}$ values and the
cdf of a $\chi_{(4)}^{2}$ distribution. (Right) The histogram of
$\mathcal{D}$ values and the pdf of a $\chi_{(4)}^{2}$ distribution.}
\end{figure}

\end{document}